\documentclass[reqno]{amsart}

\usepackage{amsmath}
\usepackage{amsfonts}
\usepackage{amssymb,enumerate}
\usepackage{amsthm}
\usepackage[all]{xy}
\usepackage{rotating}
\usepackage{hyperref}
\usepackage{color}

\theoremstyle{plain}
\newtheorem{lem}{Lemma}[section]
\newtheorem{cor}[lem]{Corollary}
\newtheorem{prop}[lem]{Proposition}

\newtheorem*{mthm*}{Main Theorem}
\newtheorem*{mcor*}{Corollary}

\theoremstyle{definition}

\newtheorem{disc}[lem]{Remark}

\newtheorem{para}[lem]{}

\newtheorem*{convention*}{Convention}

\newcommand{\pd}{\operatorname{pd}}

\newcommand{\id}{\operatorname{id}}

\newcommand{\Ht}{\operatorname{ht}}	
	
\newcommand{\depth}{\operatorname{depth}}	
\newcommand{\rank}{\operatorname{rank}}

\newcommand{\spec}{\operatorname{Spec}}

\newcommand{\ideal}[1]{\mathfrak{#1}}

\newcommand{\fm}{\ideal{m}}

\newcommand{\fq}{\ideal{q}}
\newcommand{\fa}{\ideal{a}}

\newcommand{\xra}{\xrightarrow}
\newcommand{\xla}{\xleftarrow}

\renewcommand{\geq}{\geqslant}
\renewcommand{\leq}{\leqslant}

\newcommand{\Ext}[4][R]{\operatorname{Ext}_{#1}^{#2}(#3,#4)}

\newcommand{\Hom}{\operatorname{Hom}}

\def\Ext{\operatorname{Ext}}

\numberwithin{equation}{lem}

\begin{document}

\bibliographystyle{amsplain}

\title[On Gorenstein fiber products And Applications]{On Gorenstein fiber products And Applications}

\author{Saeed Nasseh}

\address{Department of Mathematical Sciences,
Georgia Southern University,
Statesboro, Georgia 30460, U.S.A.}

\email{snasseh@georgiasouthern.edu}
\urladdr{https://cosm.georgiasouthern.edu/math/saeed.nasseh}

\author{Ryo Takahashi}
\address{Graduate School of Mathematics\\
Nagoya University\\
Furocho, Chikusaku, Nagoya, Aichi 464-8602, Japan}
\email{takahashi@math.nagoya-u.ac.jp}
\urladdr{http://www.math.nagoya-u.ac.jp/~takahashi/}

\author{Keller VandeBogert}

\address{Department of Mathematical Sciences,
Georgia Southern University,
Statesboro, Georgia 30460, U.S.A.}

\email{keller\_l\_vandebogert@georgiasouthern.edu}

\thanks{Takahashi was partly supported by JSPS Grants-in-Aid
for Scientific Research 16K05098.}

\date{\today}



\begin{abstract}
We show that a non-trivial fiber product $S\times_k T$ of commutative noetherian local rings $S,T$ with a common residue field $k$ is Gorenstein if and only if it is a hypersurface of dimension 1. In this case, both $S$ and $T$ are regular rings of dimension 1. We also give some applications of this result.
\end{abstract}

\maketitle

\section{Introduction}\label{sec161010a}

Throughout this paper, $(S,\fm_S,k)$ and $(T,\fm_T,k)$ are commutative noetherian local rings with a common residue field $k$, and $S\times_k T$ denotes the fiber product of $S$ and $T$ over $k$. Note that, $S\times_k T$ is the pull-back of the natural surjections $S\xra{\pi_S} k\xla{\pi_T}T$ and
$$
S\times_k T=\left\{(s,t)\in S\times T\mid \pi_S(s)=\pi_T(t)\right\}.
$$
This ring is a commutative noetherian local ring with maximal ideal $\fm_{S\times_k T}=\fm_S\oplus \fm_T$ and residue field $k$. Also, $\fm_S$ and $\fm_T$ are ideals of $S\times_k T$ and there are ring isomorphisms $S\cong (S\times_k T)/\frak m_T$ and $T\cong (S\times_k T)/\frak m_S$. If $S\neq k\neq T$ (or equivalently, $\fm_S\neq 0\neq \fm_T$), then we say that $S\times_k T$ is a non-trivial fiber product. (For more information about fiber products, in addition to the references introduced below, see~\cite{christensen:gmirlr, dress, kostrikin, lescot:sbpfal, moore, NSW, ogoma, ogoma1}.)

In case that $S=T$, it is shown in~\cite[Theorem 1.8]{AAM} that $S\times_k S$ is Gorenstein if and only if $S$ is a regular ring of dimension 1. (See also~D'Anna~\cite{D'Anna} and Shapiro~\cite{Shapiro}.) In this note we give the following generalization of~\cite[Theorem 1.8]{AAM} which we prove in~\ref{para161012a}; compare this result with Ogoma~\cite[Theorem 4]{Ogoma2}.

\begin{mthm*}\label{cor161010b}
Let $S\times_k T$ be a non-trivial fiber product.
The ring $S\times_k T$ is Gorenstein if and only if it is a hypersurface of dimension 1, and then both $S$ and $T$ are regular rings of dimension 1.

Moreover, in this case, $S\times_k T$ is isomorphic to a fiber product $Q/(p)\times_k Q/(q)$, where $Q$, $Q/(p)$, and $Q/(q)$ are regular local rings with residue field $k$ and $p,q\in Q$ are prime elements.
\end{mthm*}

As applications of this theorem, we give a stronger version of a result of Takahashi~\cite[Theorem A]{Ryo1} and prove a generalization of~\cite[Proposition 3.10]{adela} due to Atkins and Vraciu; see Corollaries~\ref{cor170106b} and~\ref{cor160519a}.

\section{Proof Of Main Theorem}\label{sec161010b}

For the rest of this paper, $(R,\fm_R, \ell)$ will be a commutative noetherian local ring, and recall that the rings $S$ and $T$ are introduced in the Introduction.
The following result is proved in~\cite[Proposition 1.7]{AAM} when $S$ and $T$ are artinian.

\begin{prop}\label{prop170110a}
Let $S\times_k T$ be non-trivial fiber product. Then
$S\times_kT$ is Cohen-Macaulay if and only if $\dim S\times_kT\leq 1$ and $S$ and $T$ are Cohen-Macaulay with $\dim S=\dim S\times_kT=\dim T$.
\end{prop}

\begin{proof}
The assertion follows from the equalities
\begin{align}
\dim S\times_kT&=\max\{\dim S,\dim T\}\notag \\
\depth S\times_kT&=\min\{\depth S,\depth T,1\}.\notag
\end{align}
(See Lescot~\cite{lescot:sbpfal}, or Christensen, Striuli, and Veliche \cite[Remark 3.1]{christensen:gmirlr}.)
\end{proof}

The following lemma will be used in the proof of Main Theorem.

\begin{lem}\label{lem170106a}
Assume that $R$ is a hypersurface of dimension 1, and let $I$ be a non-zero ideal of $R$ such that $R/I$ is a regular ring of dimension 1. Then $R/I\cong Q/(f)$, where $Q$ is a regular local ring and $f\in Q$ is a prime element.
\end{lem}

\begin{proof}
Let $R\cong Q/(g)$, where $(Q,\fm_Q,\ell)$ is a 2-dimensional regular local ring and $g\in \fm_Q$. Since $I$ is a prime ideal of $R$, it corresponds to a prime ideal $\fq/(g)$ of $Q/(g)$, where $\fq\in \spec(Q)\cap V((g))$. If $g=g_1g_2\cdots g_n$ is a prime factorization of $g$ in $Q$, then there exists an integer $1\leq i\leq n$ such that $g_i\in \fq$. Let $f:=g_i$, and note that $\Ht_Q(\fq)=1=\Ht_Q((f))$ because $\Ht_R(I)=0$. Hence, $\fq=(f)$.
Therefore,
$$
\frac{R}{I}\cong \frac{Q/(g)}{\fq/(g)}=\frac{Q/(g)}{(f)/(g)}\cong \frac{Q}{(f)}
$$
as desired.
\end{proof}

Next we introduce some notations and discuss some results from~\cite{kostrikin} and~\cite{lescot:sbpfal} to use in the proof of our Main Theorem. (See also~\cite{christensen:gmirlr}.)

\begin{para}
Let $M$ be a finitely generated $R$-module. Recall that the \emph{Poincar\'e series} and the \emph{Bass series} of $M$, denoted $P^M_R(t)$ and $I_M^R(t)$, respectively, are the formal Laurent series defined as follows:
\begin{align*}
P^M_R(t)&:=\sum_{i\geq 0}\rank_{\ell}(\Ext_R^i(M,\ell))t^i\\
I_M^R(t)&:=\sum_{i\geq 0}\rank_{\ell}(\Ext_R^i(\ell,M))t^i.
\end{align*}
We simply denote $I_R^R(t)$ by $I_R(t)$.
The coefficient of $t^{\depth R}$ in $I_R(t)$ is called \emph{type of $R$}, and is denoted $\gamma_R$. Note that $\gamma_R\neq 0$ and all the coefficients of $t^i$ in $I_R(t)$ for $i<\depth R$ are zero. Also, note that the constant term in $P^{\ell}_R(t)$ is 1.
\end{para}

\begin{para}
By Kostrikin and {\v{S}}afarevi{\v{c}}~\cite{kostrikin} we have the equality
\begin{equation}\label{eq160315e}
\frac{1}{P_{S\times_kT}^k(t)}=\frac{1}{P_S^{k}(t)}+\frac{1}{P_T^{k}(t)}-1
\end{equation}
which gives a relation between Poincar\'{e} series of $k$ over $S\times_kT$ and over the rings $S$ and $T$.
Also, by a result of Lescot~\cite[Theorem 3.1]{lescot:sbpfal} we have the following formulas:

If $S$ and $T$ are singular, then
\begin{equation}\label{eq160518a}
\frac{I_{S\times_kT}(t)}{P_{S\times_kT}^k(t)}=t+\frac{I_S(t)}{P_S^{k}(t)}+\frac{I_T(t)}{P_T^{k}(t)}.
\end{equation}

If $S$ is singular and $T$ is regular with $\dim T=n$, then
\begin{equation}\label{eq160518b}
\frac{I_{S\times_kT}(t)}{P_{S\times_kT}^k(t)}=t+\frac{I_S(t)}{P_S^{k}(t)}-\frac{t^{n+1}}{(1+t)^n}.
\end{equation}

If $S$ and $T$ are regular with $\dim S=m$ and $\dim T=n$, then
\begin{equation}\label{eq160518c}
\frac{I_{S\times_kT}(t)}{P_{S\times_kT}^k(t)}=t-\frac{t^{m+1}}{(1+t)^m}-\frac{t^{n+1}}{(1+t)^n}.
\end{equation}
\end{para}

We are now ready to prove the Main Theorem.

\begin{para}[\emph{Proof of Main Theorem}]\label{para161012a}
Assume that $A:=S\times_k T$ is a Gorenstein ring. By Proposition~\ref{prop170110a}, we have $\dim A\leq 1$ and $S$ and $T$ are Cohen-Macaulay with $\dim S=\dim A=\dim T$.
We prove the theorem by considering the following three cases, and when using the Poincar\'e and Bass series, we simply write $I$ and $P$ instead of $I(t)$ and $P(t)$.

Case 1: Assume that $S$ and $T$ are singular. Then by~\eqref{eq160315e} and~\eqref{eq160518a} we have
\begin{equation}\label{eq160518f}
I_A\left(P_T^k+P_S^k-P_T^kP_S^k\right)=
tP_T^kP_S^k+I_SP_T^k+I_TP_S^k.
\end{equation}

If $\dim A=0$, then both $S$ and $T$ are Cohen-Macaulay of dimension zero. Now by looking at the constant terms on the left and right of~\eqref{eq160518f} we obtain
$1=\gamma_A=\gamma_S+\gamma_T$. But this is impossible because $\gamma_S$ and $\gamma_T$ are positive integers.

If $\dim A=1$, then $S$ and $T$ are Cohen-Macaulay of dimension one. Now by looking at the coefficient of $t$ on the left and right of~\eqref{eq160518f} we obtain $1=\gamma_A=1+\gamma_S+\gamma_T$. Hence, $\gamma_S+\gamma_T=0$, which is again impossible.
Therefore, both of $S$ and $T$ cannot be singular, and Case 1 does not hold.

Case 2: Assume that $S$ is singular and $T$ is regular with $\dim T=n$.
Then it follows from~\eqref{eq160315e} and~\eqref{eq160518b} that
\begin{equation}\label{eq160518g}
I_A\left(P_T^k+P_S^k-P_T^kP_S^k\right)(1+t)^n=
\left(t(1+t)^n-t^{n+1}\right)P_T^kP_S^k+(1+t)^nI_SP_T^k.
\end{equation}

If $\dim A=0$, then we have $n=0$. Since $T$ is regular, we must have $T=k$, which is a contradiction with our assumption.

If $\dim A=1$, then $n=1$. Now by looking at the coefficient of $t$ on the left and right of~\eqref{eq160518g} we obtain $1=\gamma_A=1+\gamma_S$. This implies that $\gamma_S=0$, which is impossible. Hence, Case 2 also does not hold.

Therefore, the only possibility is the following case.

Case 3. Both $S$ and $T$ are regular rings.
If $\dim A=0$, then both $S$ and $T$ have dimension zero, and hence, both are equal to $k$.
This contradiction shows that we must have $\dim A=1=\dim S=\dim T$. Therefore, by~\cite[(3.2) Observation]{christensen:gmirlr}, the ring $A$ is a hypersurface of dimension one.

For the second part of the Main Theorem, note that $S\cong A/\fm_T$ and $T\cong A/\fm_S$. Hence the assertion follows from Lemma~\ref{lem170106a} and its proof. \qed
\end{para}

We conclude this section with the following result that will be used later.

\begin{prop}\label{prop170110b}
A non-trivial fiber product $A:=S\times_k T$ is not regular.
\end{prop}

\begin{proof}
If $A$ is a regular ring, then by Proposition~\ref{prop170110a} we have $\dim A\leq 1$. Now by the Auslander-Buchsbaum formula we have $\pd_A (A/\fm_T)\leq 1$. This implies that $\pd_A(\fm_T)=0$, and hence $\fm_T$ is a free $A$-module. But this cannot happen because $\fm_S\fm_T=0$, and $\fm_S\neq 0$. Therefore, $A$ is not a regular ring.
\end{proof}

\section{Applications}

This section contains some applications of the Main Theorem. In particular, we give a stronger version of a result of Takahashi and prove a generalization of a result of Atkins and Vraciu; see Corollaries~\ref{cor170106b} and~\ref{cor160519a} below.

We start with a result of Ogoma~\cite[Lemma 3.1]{ogoma} that plays an essential role in this section.

\begin{para}\label{disc161010a}
Let $\fa\subseteq R$ be an ideal of $R$ that has a decomposition $\fa=I\oplus J$, where $I$ and $J$ are non-zero ideals of $R$. Then there is an isomorphism
$R\cong (R/I)\times_{R/\fa}(R/J)$.
This isomorphism is naturally defined by
$r\mapsto (r+I,r+J)$ for $r\in R$.
\end{para}

As an immediate observation of this discussion we record the following result.

\begin{prop}\label{prop161010a}
A local ring is a non-trivial fiber product of the form $S\times_k T$ if and only if its maximal ideal is decomposable.
\end{prop}

From Proposition~\ref{prop170110b} we obtain the following result.

\begin{cor}\label{cor170119a}
If $\fm_R$ is decomposable, then $R$ is not regular.
\end{cor}

The next result follows directly from~\cite[Corollary 4.2]{NSW}.

\begin{cor}\label{cor161013a}
Assume that $\fm_R$ is decomposable. For finitely generated $R$-modules $M$ and $N$ if $\Ext^{i}_R(M,N)=0$ for all $i\gg 0$, then $\pd_R(M)\leq 1$ or $\id_R(N)\leq 1$.
\end{cor}

\begin{disc}\label{disc170119a}
Corollary~\ref{cor161013a} shows in particular that if $\fm_R$ is decomposable, then $R$ satisfies the Auslander-Reiten Conjecture, that is, if for a finitely generated $R$-module $M$ we have $\Ext^{i}_R(M,M\oplus R)=0$ for all $i>0$, then $M$ is a free $R$-module. (See~\cite{auslander:gvnc} for details about this conjecture.)
\end{disc}

The following is a stronger version of a result of Takahashi~\cite[Theorem A]{Ryo1}. Note that in this corollary we do not assume that $R$ is complete; see Remark~\ref{disc050117}.

\begin{cor}\label{cor170106b}
If $\fm_R$ is decomposable, then the following are equivalent.
\begin{enumerate}[\rm(i)]
\item
There is a finitely generated $R$-module $E$ of finite injective dimension such that $\Ext^{i}_R(E,R)=0$ for all $i\gg 0$;
\item
$R$ is Gorenstein;
\item
$R$ is a hypersurface of dimension 1. In this case, $R$ is isomorphic to a fiber product $Q/(p)\times_{\ell} Q/(q)$, where $Q$, $Q/(p)$, and $Q/(q)$ are regular local rings with residue field $\ell$ and $p,q\in Q$ are prime elements;
\item
There is a finitely generated $R$-module $M$ with infinite projective dimension such that $\Ext^{i}_R(M,R)=0$ for all $i\gg 0$.
\end{enumerate}
\end{cor}

\begin{proof}
By Proposition~\ref{prop161010a}, the ring $R$ is a non-trivial fiber product. Let us assume that $R=S\times_k T$. (In particular, $\ell=k$ in this case.)

(i)$\implies$(ii). It follows from our vanishing assumption and Corollary~\ref{cor161013a} that $R$ is Gorenstein or $\pd_R(E)<\infty$. In the latter case, $R$ is also Gorenstein by Foxby~\cite{foxby}.

(ii)$\implies$(iii) follows directly from the Main Theorem.

(iii)$\implies$(iv). By Corollary~\ref{cor170119a}, the ring $R$ is not regular. Thus, by Auslander-Buchsbaum and Serre~\cite{auslander:cm, Serre} we have $\pd_R(k)=\infty$. Since $R$ is Gorenstein we also have $\Ext^{i}_R(k,R)=0$ for all $i\gg 0$.

(iv)$\implies$(i). Since $M$ has infinite projective dimension, our vanishing assumption and Corollary~\ref{cor161013a} imply that $R$ is a Gorenstein ring. This completes the proof.
\end{proof}

\begin{disc}\label{disc050117}
In Corollary~\ref{cor170106b}, if we further assume that $R$ is a
quotient of a regular ring, then by~\cite[Theorem 3.2.4]{Ryo1} the ring $R$ is isomorphic to a quotient $A/(xy)$ of a regular local ring $A$ of dimension 2, where $x,y$ is a regular system of parameters for $A$.
\end{disc}

The following is a generalization of~\cite[Proposition 3.10]{adela}. Recall that a finitely generated $R$-module $X$ is \emph{totally reflexive} if
$\Hom_R(\Hom_R(X,R),R)\cong X$ and
$\Ext^{i}_R(X,R)=0=\Ext^{i}_R(\Hom_R(X,R),R)$ for all $i> 0$.

\begin{cor}\label{cor160519a}
Assume that $\fm_R$ is decomposable. If $R$ is artinian, then $R$ has no non-free finitely generated module $M$ such that $\Ext^{i}_R(M,R)=0$ for all $i\gg 0$. In particular, $R$ has no non-free totally reflexive modules.
\end{cor}

\section*{Acknowledgments} We are grateful to Mohsen Asgharzadeh, Ananthnarayan Hariharan, and Sean Sather-Wagstaff for helpful discussions about this work.


\providecommand{\bysame}{\leavevmode\hbox to3em{\hrulefill}\thinspace}
\providecommand{\MR}{\relax\ifhmode\unskip\space\fi MR }
\providecommand{\MRhref}[2]{%
  \href{http://www.ams.org/mathscinet-getitem?mr=#1}{#2}
}
\providecommand{\href}[2]{#2}

\end{document}